\documentclass[9pt]{article}
\usepackage{amsmath,amsthm,amssymb,amscd,fancybox}
\usepackage{ascmac}
\usepackage[arrow,matrix]{xy}
\usepackage{enumerate}
\usepackage{mathrsfs} 
\usepackage{graphicx, color}
\usepackage{tikz}
\usepackage{here} 
\date{}
\author{}

\theoremstyle{definition}
\numberwithin{equation}{section}

\newtheorem{thm}{Theorem}[section]
\newtheorem{dfn}[thm]{Definition}
\newtheorem{exa}[thm]{Example}
\newtheorem{prop}[thm]{Proposition}
\newtheorem{cor}[thm]{Corollary}
\newtheorem{lem}[thm]{Lemma}

\newtheorem{rem}[thm]{Remark}

\newcommand{\mf}[1]{{\mathfrak{#1}}}
\newcommand{\mb}[1]{{\mathbf{#1}}}
\newcommand{\bb}[1]{{\mathbb{#1}}}
\newcommand{\mca}[1]{{\mathcal{#1}}}

\usepackage[utf8]{inputenc}
\usepackage[T1]{fontenc}

\title{Nested Affine Buildings and Their Group Decompositions.}

\author{Masaoki Mori \thanks{u602050e@alumni.osaka-u.ac.jp, masaokimori12ib@gmail.com}}
\date{}

\begin{document}

\maketitle

\begin{abstract}
In this paper, we construct a higher dimensional generalization of affine buildings and introduce a new structure, which we call \textit{Babel buildings}. These buildings are non-connected, non-convex metric spaces of non-positive curvature. Despite their non-standard properties, Babel buildings provide an effective framework for studying the structure of groups acting on them. We analyze the metric and nesting structures of Babel buildings and derive key results regarding the group actions consistent with this new framework.
\end{abstract}



\section*{Introduction}

The theory of Bruhat-Tits buildings has established itself as a fundamental framework for understanding the structure of reductive algebraic groups over local fields. In their monumental work \cite{BT}, Bruhat and Tits constructed affine buildings as geometric analogues of symmetric spaces. A crucial feature of an affine building is that it carries a canonical metric, making it a CAT(0) space (a complete metric space of non-positive curvature). This geometric property leads to powerful results, such as the fixed point theorem for compact group actions and various decomposition theorems for the associated groups.

In recent years, there has been increasing interest in generalizing this theory to \textit{higher-dimensional local fields}, such as the field of iterated Laurent series $k((t_1))\cdots((t_n))$. For such fields, the valuation group is no longer a subgroup of $\mathbb{R}$, but rather isomorphic to $\mathbb{Z}^n$ equipped with the lexicographic order. Consequently, the associated buildings cannot be metrized by standard real numbers in a way that preserves the full valuation structure. While the combinatorial structure of such buildings has been studied (e.g., as $\Lambda$-buildings in the sense of Bennett or Bennett-Schwer-Struyv, \cite{Ben1}, \cite{Ben2} and \cite{BSS}), a satisfactory metric theory analogous to the classical CAT(0) geometry has been less developed.

While the axiomatic theory of $\Lambda$-buildings provides a rigorous framework, it often obscures the geometric intuition of the space, especially when the valuation group is of higher rank (e.g., $\mathbb{Z}^n$). The metric structure becomes abstract, making it difficult to visualize fundamental geometric operations such as retractions or group decompositions.Motivated by this difficulty, we adopt a constructive approach. By realizing the building as a nested sequence of lower-level buildings, we recover the geometric intuition. In our model, the algebraic residue structure corresponds precisely to the hierarchical 'zooming' between different levels of the building. This concrete visualization allows us to prove results like the Kapranov decomposition, which are hard to perceive from the purely axiomatic viewpoint.

In this paper, we propose a new geometric approach to buildings for higher-dimensional local fields. Instead of projecting the geometry down to $\mathbb{R}$, we extend the concept of the metric itself. We introduce the ring of \textit{hyper-real numbers} of level $n$, denoted by ${}^{n*}\mathbb{R}$, which is essentially a lexicographic product of $n$ copies of $\mathbb{R}$. We define a class of buildings equipped with ${}^{n*}\mathbb{R}$-valued metrics, which we call \textbf{Babel buildings}. The name "Babel" reflects the hierarchical, nested structure of these spaces: an $n$-level Babel building can be viewed as an infinite tower of $(n-1)$-level Babel buildings, culminating in the classical affine buildings at level 1.

The main objective of this paper is to establish that Babel buildings possess "nice" geometric properties parallel to those of affine buildings, despite the hyper-real nature of the metric values.

Our first main result is the construction of a well-defined metric on the Babel building $X$. Since Babel buildings lack gallery connectedness for $n>1$, the existence of a geodesic metric satisfying the triangle inequality is non-trivial. By employing the method of retractions onto apartments, we prove:

\begin{thm}[Main Theorem A]
Let $X$ be an $n$-level Babel building. There exists a canonical distance function $d_X \colon X \times X \longrightarrow {}^{n*}\mathbb{R}$ which satisfies the triangle inequality. Furthermore, the retraction map onto any Babel apartment is distance-decreasing.
\end{thm}

Once the metric is established, we investigate the curvature properties of $X$. Although $X$ is not a CAT(0) space in the classical sense (as the metric is not real-valued), we formulate a generalized CAT(0)-inequality in terms of ${}^{n*}\mathbb{R}$. This allows us to generalize the celebrated Bruhat-Tits fixed point theorem.

\begin{thm}[Main Theorem B]
Let $X$ be an $n$-level Babel building. Then $X$ satisfies the ${}^{n*}\mathbb{R}$-valued CAT(0)-inequality. Consequently, if a group $G$ acts isometrically on $X$ and stabilizes a bounded subset $B$ that admits a circumcenter, then $G$ fixes the circumcenter of $B$.
\end{thm}

Finally, we apply this geometric framework to the study of group actions. Let $\widehat{G}$ be a group acting strongly transitively on $X$. We show that the geometric structure of $X$ naturally induces decompositions of $\widehat{G}$ that generalize the classical matrix decompositions.

\begin{thm}[Main Theorem C]
Let $\widehat{G}$ be a group with a strongly transitive, type-preserving action on an $n$-level Babel building. Then $\widehat{G}$ admits the following decompositions:
\begin{itemize}
    \item The \textbf{Bruhat decomposition}: $\widehat{G} = \coprod_{w \in \widehat{W}} \widehat{B} w \widehat{B}$.
    \item The \textbf{Cartan decomposition}: $\widehat{G} = \coprod_{v \in \widehat{V}_{\mathbf{D}}} K v K$.
    \item The \textbf{$(i,j)$-Kapranov decomposition}: A family of decompositions relating sectors of different levels, which includes the Iwasawa decomposition as the $(0,n)$-case.
\end{itemize}
\end{thm}

This paper is organized as follows. Section 1 is devoted to the geometry of Babel buildings, where we construct the metric and prove the metric-decreasing property of retractions and we discuss the fixed point theorem and the group decompositions. In Section 2, we discuss the relationship between our theory and the homological properties of arithmetic quotients.

\section{The structure of Babel buildings}

\subsection{Hyper-extensions of sets}

\begin{dfn}
A subset $\mca{F}$ of the power set of $\bb{N}$ is called a \textit{filter} if it satisfies the following conditions:
\begin{enumerate}[(i)]
\item $\emptyset \notin \mca{F}$;
\item If $A,B\in \mca{F},$ then $A\cap B\in \mca{F}$;
\item If $A\in \mca{F}$ and $A\subset B,$ then $B\in \mca{F}$.
\end{enumerate}
\end{dfn}

For example, the set $\{A\subset \bb{N} \mid \mathrm{Card}(A^{c})<\infty\}$ forms a filter, known as the \textit{cofinite filter}.

\begin{dfn} 
A filter $\mca{F}$ is called an \textit{ultrafilter} if it satisfies the condition:
$$
A\in \mca{F} \iff A^{c}\notin \mca{F}.
$$
It is a well-known fact that any filter is contained in some ultrafilter.
\end{dfn}

Throughout this paper, we fix an ultrafilter $\mca{F}$ containing the cofinite filter.

\begin{dfn}
\begin{enumerate}[(a)]
\item Let $A$ be a set. Two sequences $\{a_{m}\}_{m\in \bb{N}}$ and $\{b_{m}\}_{m\in \bb{N}}$ in $A$ are said to be $\mca{F}$-equivalent if the set $\{m\in \bb{N} \mid a_{m}=b_{m}\}$ belongs to $\mca{F}$. This defines an equivalence relation, denoted by $\{a_{m}\}_{m\in \bb{N}}\sim_{\mca{F}} \{b_{m}\}_{m\in \bb{N}}$. We call the quotient set
$$
{}^{*}A:=A^{\bb{N}}/\sim_{\mca{F}}
$$
the \textit{hyper-extension} of $A$ and denote the equivalence class of $\{a_{m}\}_{m\in \bb{N}}$ by $[a_{m}]$.
\item Let $f\colon A\longrightarrow B$ be a map. Then there is an induced map between the hyper-extensions:
$$
{}^{*}f\colon {}^{*}A\longrightarrow {}^{*}B; \quad [a_{m}]\longmapsto [f(a_{m})].
$$
We call this map the \textit{hyper-extension} of $f$.
\end{enumerate}
\end{dfn}

\begin{rem}
For any set $A$, we regard $A$ as a subset of its hyper-extension ${}^{*}A$ via the diagonal embedding $a\longmapsto [a,a,\ldots]$. For any map $f\colon A\longrightarrow B$, it is clear that the restriction ${}^{*}f\vert_{A}$ coincides with $f$.
\end{rem}

We have defined the hyper-extension functor ${}^{*}\colon \mathrm{Set}\longrightarrow \mathrm{Set}$. Let us denote by ${}^{n*}$ the $(n-1)$-fold composition of ${}^{*}$. Thus, we have ${}^{1*}=\mathrm{Id}$.

\begin{lem}
\begin{enumerate}[(i)]
\item Let $(A,<)$ be a totally ordered set. For any two elements $[a_{m}],[b_{m}]\in {}^{*}A$, we define $[a_{m}]<[b_{m}]$ by the condition that the set $\{m\in \bb{N}\mid a_{m}<b_{m}\}$ belongs to $\mca{F}$. Then $({}^{*}A,<)$ is a totally ordered set containing $A$ as an ordered subset.
\item Let $(A,+,\times)$ be a ring (resp. field). For any two elements $[a_{m}],[b_{m}]\in {}^{*}A$, we define
\begin{align*}
[a_{m}]+[b_{m}]&=[a_{m}+b_{m}], \\
[a_{m}]\times [b_{m}]&=[a_{m}b_{m}].
\end{align*}
Then $({}^{*}A,+,\times)$ is a ring (resp. field) and $A$ is a subring (resp. subfield) of ${}^{*}A$.
\item If $A$ is a ring and $M$ is an $A$-module, then ${}^{*}M$ is an ${}^{*}A$-module.
\end{enumerate}
\end{lem}
\begin{proof}
\begin{enumerate}[(i)]
\item For any $[a_{m}],[b_{m}]\in {}^{*}A$, let us set
\begin{align*}
J_{=}&=\{m\in \bb{N}\mid a_{m}=b_{m}\}, \\
J_{<}&=\{m\in \bb{N}\mid a_{m}<b_{m}\}, \\
J_{>}&=\{m\in \bb{N}\mid a_{m}>b_{m}\}.
\end{align*}
Since $A$ is a totally ordered set, the decomposition $\bb{N}=J_{<}\sqcup J_{=}\sqcup J_{>}$ holds. If $[a_{m}]\not< [b_{m}]$, then $J_{<}\notin \mca{F}$. Since $\mca{F}$ is an ultrafilter, this implies that $J_{=}\sqcup J_{>}$ belongs to $\mca{F}$. Thus we have $[a_{m}]\geq [b_{m}]$, which proves that ${}^{*}A$ is a totally ordered set.
\item It is clear that ${}^{*}A$ is a ring and $A$ is a subring of ${}^{*}A$ under the operations defined above. Suppose that $A$ is a field. If $[a_{m}]\in {}^{*}A$ is non-zero, then the set $J=\{m\in \bb{N}\mid a_{m}\neq 0\}$ belongs to $\mca{F}$. Let us define
$$
b_{m} = \left\{
\begin{array}{ll}
\frac{1}{a_{m}} & (m \in J)\\
0 & (m \notin J).
\end{array}
\right.
$$
Then the set $\{m\in \bb{N}\mid a_{m}b_{m}=1\}$ contains $J$ and thus lies in $\mca{F}$. Hence we have $[a_{m}]^{-1}=[b_{m}]$ in ${}^{*}A$.

\item The proof is omitted.
\end{enumerate}
\end{proof}

Since $\mca{F}$ contains the cofinite filter, there exists an element $\omega\in {}^{*}\bb{R}$ such that $\omega>r$ for all $r\in \bb{R}$. Such an element is called \textit{infinitely large}. Similarly, we say that an element $\omega_{i}\in {}^{i*}\bb{R}$ is \textit{$i$-level infinitely large} if $\omega_{i}>r$ for all $r\in {}^{(i-1)*}\bb{R}$.

Throughout this paper, we fix infinitely large elements $\omega_{i}\in {}^{i*}\bb{Z}$ of level $i$ for each $i=2,\ldots,n$, and set $\omega_{1}=1$. This yields a natural isomorphism
$$
\bb{Z}^{n}\simeq \bigoplus_{i=1}^{n}\bb{Z}\omega_{i}
$$
of totally ordered Abelian groups, where $\bb{Z}^{n}$ is equipped with the lexicographic order. Via this isomorphism, we regard $\bb{Z}^{n}$ as a subgroup of the totally ordered field ${}^{n*}\bb{R}$.

\begin{rem}
The lexicographically ordered group $\bb{R}^{n}\simeq \bigoplus_{i=1}^{n}\bb{R}\omega_{i}$ is contained in the hyper-line ${}^{n*}\bb{R}$. The vectors $\omega_{1},\ldots,\omega_{n}$ share the same direction (sign) but are linearly independent over $\bb{R}$. We call such vectors \textit{magnitude-independent}. Later, we will consider apartments consisting of valuations of a root datum for $SL_{d}(F)$, where $F$ is an $n$-dimensional local field. This space possesses $d-1=\mathrm{rank}(A_{d-1})$ independent directions, and in each direction, there are $n=\dim F$ magnitude-independent vectors.
\end{rem}

\begin{rem}
Let $A=\{a_{1},\ldots,a_{m}\}$ be a finite set. Suppose that there exists an element $x=[x_{j}]_{j\in \bb{N}}\in {}^{*}A\setminus A$. Then the sets
\begin{align*}
N_{1}&=\{j\in \bb{N}\mid x_{j}=a_{1}\}, \\
&\vdots \\
N_{m}&=\{j\in \bb{N}\mid x_{j}=a_{m}\}
\end{align*}
form a partition $\bb{N}=N_{1}\sqcup \cdots \sqcup N_{m}$. Since $x\notin A$, we have $N_{1},\ldots,N_{m}\notin \mca{F}$. However, this contradicts the fact that $\mca{F}$ is an ultrafilter (which must contain exactly one element of any finite partition). Therefore, if $A$ is a finite set, then ${}^{*}A=A$.
\end{rem}

\begin{rem}\label{Com1}
We describe a similarity between hyperreal fields and higher dimensional local fields. Let us define
$$
\mathscr{O}^{(1)}_{{}^{n*}\bb{R}}=\{x\in {}^{n*}\bb{R}\mid \exists r\in {}^{(n-1)*}\bb{R}\text{ such that }{}^{n*}|x|\leq r\}.
$$
We denote by $\mf{m}_{n}\subset \mathscr{O}^{(1)}_{{}^{n*}\bb{R}}$ the set of $n$-level infinitesimals. Then we have an isomorphism
$$
\mathscr{O}^{(1)}_{{}^{n*}\bb{R}}/\mf{m}_{n}\simeq {}^{(n-1)*}\bb{R},
$$
and ${}^{n*}\bb{R}$ is the field of fractions of $\mathscr{O}^{(1)}_{{}^{n*}\bb{R}}$. 
Let $F$ be an $n$-dimensional local field and $\mathscr{O}_{F}^{(1)}$ the discrete valuation ring of $F$. The residue field of $\mathscr{O}^{(1)}_{F}$ is an $(n-1)$-dimensional local field, and the field of fractions of $\mathscr{O}_{F}^{(1)}$ is $F$. In view of this analogy, the field ${}^{n*}\bb{R}$ may be regarded as an ``$n$-dimensional Archimedean local field".
\end{rem}

\subsection{Babel apartments}

\begin{dfn}
\begin{enumerate}[(i)]
\item A pair $(X,d_{X})$ is called an \textit{${}^{n*}\bb{R}$-metric space} if the map
$$
d_{X}\colon X\times X\longrightarrow {}^{n*}\bb{R}_{\geq 0}
$$ 
satisfies the following conditions:
\begin{enumerate}[(a)]
\item $d_{X}(x,y)=0 \iff x=y$ for all $x,y\in X$;
\item $d_{X}(x,y)=d_{X}(y,x)$ for all $x,y\in X$;
\item $d_{X}(x,z)\leq d_{X}(x,y)+d_{X}(y,z)$ for all $x,y,z\in X$.
\end{enumerate}
In particular, if $n\geq 2$, such a space is called a \textit{hyper-metric space}.
\item Let $(X,d_{X})$ and $(Y,d_{Y})$ be ${}^{n*}\bb{R}$-metric spaces. A map $\psi\colon X\longrightarrow Y$ is called \textit{isometric} if
$$
d_{X}(x,y)=d_{Y}(\psi(x),\psi(y))
$$
holds for all $x,y\in X$.  
\end{enumerate}
\end{dfn}

\begin{exa}
For example, the $n$-dimensional local field $F=\bb{F}_{q}((t_{1}))\ldots ((t_{n}))$ carries an ${}^{n*}\bb{R}$-metric defined by
$$
d(x,y) = \left( \frac{1}{q}\right)^{v_{F}(x-y)}.
$$
Here, $v_{F}$ denotes the rank-$n$ valuation on $F$ associated with the system of local parameters $t_{1},\ldots,t_{n}$.
\end{exa}

Let $V$ be a finite-dimensional real vector space equipped with a positive definite inner product
$$
(\cdot,\cdot)\colon V\times V\longrightarrow \bb{R}.
$$
Let $\Phi$ be a root system in $V$. Since $\Phi$ is finite, we have ${}^{n*}V={}^{n*}\bb{R}\otimes_{\bb{R}} V$ (or simply ${}^{n*}V \simeq ({}^{n*}\bb{R})^{\dim V}$). Let us consider the hyper-extension of the inner product:
$$
{}^{n*}(\cdot,\cdot)\colon {}^{n*}V\times {}^{n*}V\longrightarrow {}^{n*}\bb{R}.
$$
This form remains a positive definite inner product over the field ${}^{n*}\bb{R}$.

\begin{dfn}
For any $a\in \Phi$ and $k\in \bb{Z}^{n}$, we define a \textit{hyper-affine reflection} $s_{a,k}$ by
$$
s_{a,k}\colon {}^{n*}V\longrightarrow {}^{n*}V; \quad v\longmapsto v-2\frac{({}^{n*}(a,v)-k)}{(a,a)}a.
$$
The group $W_{n}(\Phi)$ generated by $\{s_{a,k} \mid a\in \Phi, k\in \bb{Z}^{n}\}$ is called the \textit{$n$-level Weyl group of type $\Phi$}.
\end{dfn}

Let $\Phi^{\vee}$ denote the set $\{\frac{2a}{(a,a)} \mid a\in \Phi\}$. Then it is clear that
$$
W_{n}(\Phi)=W(\Phi)\ltimes \bb{Z}^{n}(\Phi^{\vee}).
$$
Note that when $n>1$, the Weyl group $W_{n}(\Phi)$ is a non-Coxeter reflection group.

The following propositions are hyper-analogs of well-known classical results. We therefore omit the proofs.

\begin{prop}\label{HypermetricofBA}
\begin{enumerate}[(i)]
\item For any $x,y\in {}^{n*}V$, we define
$$
{}^{n*}d_{V}(x,y)={}^{n*}(x-y,x-y)^{\frac{1}{2}}.
$$
Then $({}^{n*}V,{}^{n*}d_{V})$ is an ${}^{n*}\bb{R}$-metric space, and the action of $W_{n}(\Phi)$ on this space is isometric.
\item Let $x_{0},\ldots,x_{d}$ be affinely independent points in ${}^{n*}V$. If points $x,y\in {}^{n*}V$ satisfy
$$
{}^{n*}d_{V}(x_{i},x)={}^{n*}d_{V}(x_{i},y)
$$
for all $i=0,\ldots,d$, then $x=y$.
\item For any $t\in {}^{n*}[0,1]$ and any $x,y,z\in {}^{n*}V$, let $p_{t}=(1-t)x+ty$. Then the following equality holds:
$$
{}^{n*}d_{V}^{2}(p_{t},z)=(1-t){}^{n*}d_{V}^{2}(x,z)+t{}^{n*}d_{V}^{2}(y,z)-t(1-t){}^{n*}d_{V}^{2}(x,y).
$$
\end{enumerate}
\end{prop}

\begin{dfn}\label{defofBA}
The fundamental chamber
$$
C_{0}=\{v\in V \mid 0<(a,v)<1 \text{ for all }a\in \Phi^{+}\}
$$
of $V$ is naturally regarded as a subset of ${}^{n*}V$. Let us define
$$
\Sigma(n,\Phi)_{\bb{Z}}=\bigcup_{w\in W_{n}(\Phi)}w\overline{C_{0}}.
$$
An ${}^{n*}\bb{R}$-metric space isometric to $\Sigma(n,\Phi)_{\bb{Z}}$ is called an \textit{$n$-level Babel apartment} of type $\Phi$. We often omit the subscript and simply write $\Sigma(n,\Phi)=\Sigma(n,\Phi)_{\bb{Z}}$.
\end{dfn}

\begin{rem}
Similarly, we consider the subset $\Sigma(n,\Phi)_{\bb{R}}=\bb{R}^{n}(\Phi)$ of ${}^{n*}V$. This set is stable under the natural action of $W(\Phi)\ltimes \bb{R}^{n}(\Phi^{\vee})$. A space isometric to $\Sigma(n,\Phi)_{\bb{R}}$ is also referred to as a \textit{Babel apartment}. In future work, particularly when investigating the higher-dimensional analog of \cite[Section 7]{BT}, we will study the space $\bigoplus_{i=1}^{n}\bb{R}\omega_{i}(\Phi)$.
\end{rem}

In this paper, we restrict our attention to the Babel apartments given in Definition \ref{defofBA}. Note that the metric on a Babel apartment inherits the geometric properties established in Proposition \ref{HypermetricofBA}. In particular, for any two $n$-level Babel apartments $A_{1}, A_{2}$ of the same type and any chambers $C_{1}\subset A_{1}, C_{2}\subset A_{2}$, there exists a unique $W_{n}(\Phi)$-isometry $\varphi\colon A_{1}\longrightarrow A_{2}$ such that $\varphi(C_{1})=C_{2}$.

\begin{dfn}
Let $\varphi\colon \Sigma(n,\Phi)\xrightarrow{\sim} A$ be an isometry. We define the set $E_{A}$ to be ${}^{n*}V$. The inclusion map $A\hookrightarrow E_{A}$ is defined as the composition of the inverse isometry $\varphi^{-1}\colon A \longrightarrow \Sigma(n,\Phi)$ and the natural inclusion $\Sigma(n,\Phi) \hookrightarrow {}^{n*}V$. We call that $E_{A}$ a hyper-affine enveloping space of $A$.
\end{dfn}

\begin{exa}
\begin{enumerate}
\item Let us consider a $2$-level Babel apartment of type $A_{1}$. Following Parshin, the group $W_{2}(A_{1})$ is presented as
$$
W_{2}(A_{1})=\langle s,w_{1},w_{2} \mid s^{2}=w_{1}^{2}=w_{2}^{2}=(sw_{1}w_{2})^{2}=1\rangle.
$$
The canonical action of this group on the hyper-line ${}^{n*}\bb{R}$ is defined by
$$
s(x)=-x, \quad w_{1}(x)=2\omega_{1}-x, \quad w_{2}(x)=2\omega_{2}-x.
$$
Therefore, the geometric structure of the Babel apartment $\Sigma(2,A_{1})$ is illustrated as follows:

\begin{figure}[H]
\centering
\begin{tikzpicture}
 \draw[semithick] (-1.5,0)--(1.8,0);
 \draw[dashed] (-2.3,0)--(-1.5,0);
 \draw[dashed] (2,0)--(3.1,0);
 \draw[semithick] (3.1,0)--(5,0);
 \fill[black](0,0) circle (0.06);
 \draw (0.2,0) node[above]{$0$};
 \fill[black](0.8,0) circle (0.06);
 \draw (1,0) node[above]{$1$};
 \fill[black](2.5,0) circle (0.06);
 \draw (2.7,0) node[above]{$\omega_{2}$};
 \draw[dotted] (2.5,0.5)--(2.5,-0.5);
 \draw[dotted] (0,0.5)--(0,-0.5);
 \draw[dotted] (0.8,0.5)--(0.8,-0.5);
 \draw[<->] (0.3,-0.7) to [out=-135,in=-45] (-0.3,-0.7);
 \draw (0,-0.8) node[below]{$s$};
 \draw[<->] (1.1,-0.7) to [out=-135,in=-45] (0.5,-0.7);
 \draw (0.8,-0.8) node[below]{$w_{1}$};
 \draw[<->] (2.2,-0.7) to [out=-45,in=-135] (2.8,-0.7);
 \draw (2.5,-0.8) node[below]{$w_{2}$};
\end{tikzpicture}
\caption{The Babel apartment $\Sigma(2,A_{1})$ in ${}^{2*}\bb{R}$.}
\label{fig:A1_apt}
\end{figure}
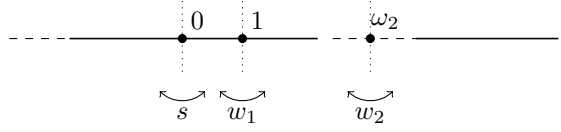

Here, each connected component is isomorphic to the affine line $\bb{R}=\bigcup_{w\in W_{a}(\Phi)}w\overline{C_{0}}$.
Thus, we have the decomposition $\Sigma(2,A_{1})=\coprod_{n\in \bb{Z}}(2n\omega_{2}+\bb{R})$ in ${}^{*}\bb{R}$.

\item Let $a=(2,0)$ and $b=(1,-\sqrt{3})$. Then we have the positive root system $A_{2}^{+}=\{a,b,a+b\}$. Similar to the previous example, the Babel apartment $\Sigma(2,A_{2})$ can be visualized as follows:

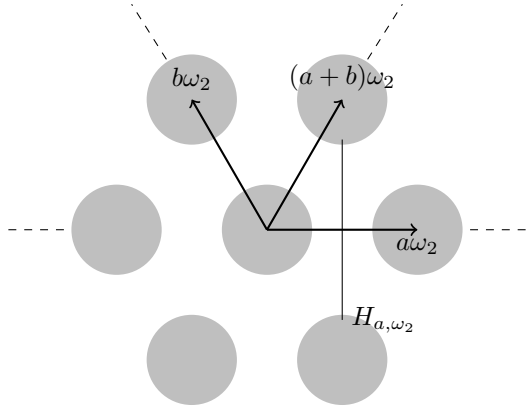
\begin{figure}[H]
\centering
\begin{tikzpicture}
 \draw[dashed] (2.7,0)--(3.5,0); 
 \draw[dashed] (1,1.732)--(1.8,3);
 \draw[dashed] (-1,1.732)--(-1.8,3);
 \draw[dashed] (-2.7,0)--(-3.5,0);
 \fill[lightgray] (0,0) circle (0.6);
 \fill[lightgray] (2,0) circle (0.6);
 \fill[lightgray] (-1,1.732) circle (0.6);
 \fill[lightgray] (1,1.732) circle (0.6);
 \fill[lightgray] (-2,0) circle (0.6);
 \fill[lightgray] (1,-1.732) circle (0.6);
 \fill[lightgray] (-1,-1.732) circle (0.6);
 \draw[->,thick] (0,0)--(2,0);
 \draw[->,thick] (0,0)--(-1,1.732); 
 \draw[->,thick] (0,0)--(1,1.732);
 \draw (2,0)node[below]{$a\omega_{2}$};
 \draw (1,1.732)node[above]{$(a+b)\omega_{2}$};
 \draw (-1,1.732)node[above]{$b\omega_{2}$};
 \draw (1,1.2)--(1,-1.2)node[right]{$H_{a,\omega_{2}}$};
\end{tikzpicture}
\caption{The Babel apartment $\Sigma(2,A_{2})$ in ${}^{2*}\bb{R}^{2}$. Each shaded region represents a copy of the standard affine plane $\bb{R}^{2}$.}
\label{fig:A2_apt}
\end{figure}

\item Let $a=(2,0)$ and $b=(-2,2)$. Then $B_{2}^{+}=\{a,b,a+b,2a+b\}$. The structure of $\Sigma(2,B_{2})$ is illustrated below:

\begin{figure}[H]
\centering
\begin{tikzpicture}
 \fill[lightgray] (0,0) circle (0.6);
 \fill[lightgray] (0,2) circle (0.6);
 \fill[lightgray] (2,0) circle (0.6);
 \fill[lightgray] (2,2) circle (0.6);
 \fill[lightgray] (0,-2) circle (0.6);
 \fill[lightgray] (-2,0) circle (0.6);
 \fill[lightgray] (-2,-2) circle (0.6);
 \fill[lightgray] (-2,2) circle (0.6);
 \fill[lightgray] (2,-2) circle (0.6);
 \draw[dashed] (-2.7,0)--(-3.5,0);
 \draw[dashed] (2.7,0)--(3.5,0);
 \draw[dashed] (-2.7,2)--(-3.5,2);
 \draw[dashed] (2.7,2)--(3.5,2);
 \draw[dashed] (-2.7,-2)--(-3.5,-2);
 \draw[dashed] (2.7,-2)--(3.5,-2);
 \draw[->,thick] (0,0)--(2,0);
 \draw[->,thick] (0,0)--(2,2);
 \draw[->,thick] (0,0)--(0,2);
 \draw[->,thick] (0,0)--(-2,2);
 \draw (2,0)node[below]{$a\omega_{2}$};
 \draw (2,2)node[above]{$(2a+b)\omega_{2}$};
 \draw (0,2)node[above]{$(a+b)\omega_{2}$};
 \draw (-2,2)node[above]{$b\omega_{2}$};
\end{tikzpicture}
\caption{The Babel apartment $\Sigma(2,B_{2})$ in ${}^{2*}\bb{R}^{2}.$}
\label{fig:B2_apt}
\end{figure}
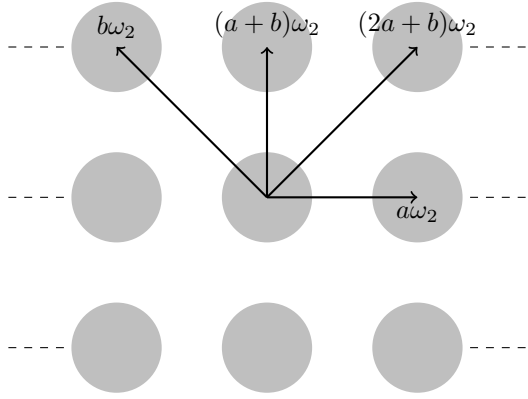
\end{enumerate}
\end{exa}

These illustrations help us to visualize the geometry of Babel apartments. Unlike the discrete case $\Sigma(n,\Phi)$, it is difficult to depict the real version $\Sigma(n,\Phi)_{\bb{R}}$ faithfully. Nevertheless, the geometric intuition remains qualitatively similar. It is important to note that for $n \ge 2$, Babel apartments are not convex. This lack of convexity constitutes a fundamental difference from the classical case of $n=1$.

\begin{rem}
We make the following observations regarding the structure of Babel apartments.
\begin{enumerate}[(i)]
\item Considered as an additive group, the Babel apartment $\Sigma(n,\Phi)$ is isomorphic to $\bb{Z}^{n-1}(\Phi)\oplus \bb{R}(\Phi)$.
\item Let $A$ be an $n$-level Babel apartment. For any $x\in A$, the subset
$$
\{y\in A \mid d_{A}(x,y)\in {}^{(n-1)*}\bb{R}\}
$$
constitutes an $(n-1)$-level Babel apartment. This property reflects the nesting structure of the Babel apartment $A$.
\end{enumerate}
\end{rem}

\begin{dfn}\label{midvecCham}
For any $i=1,\ldots,n$, consider the subgroup $W(\Phi)\ltimes \bigoplus_{j>i}\bb{Z}\omega_{j}(\Phi^{\vee})$ of $W_{n}(\Phi)$. Let $\mathbf{D}_{i}$ be a fundamental domain for this subgroup in $\Sigma(n,\Phi)$. A subset of the form $x+\mathbf{D}_{i}$ in $\Sigma(n,\Phi)$, for some $x\in \Sigma(n,\Phi)$, is called an \textit{$i$-level Babel sector}. For an $i$-level Babel sector $\mf{C}=x+\mathbf{D}_{i}$, we call $\mathbf{D}_{i}$ the \textit{direction} of $\mf{C}$ and $x$ the \textit{apex} (or \textit{cone point}) of $\mf{C}$.
\end{dfn}

In general, $i$-level Babel sectors can be defined analogously for any $n$-level Babel apartment. For convenience, we refer to a $1$-level Babel sector as an \textit{affine sector}, and a $0$-level Babel sector as a \textit{chamber}. It is evident that the non-empty intersection of two $i$-level Babel sectors with the same direction is also an $i$-level Babel sector.

\begin{rem}
Let $A$ be an $n$-level Babel apartment. Let $\mf{C}$ and $\mf{C}'$ be $i$-level Babel sectors in $A$, with apices $x$ and $x'$ respectively. If the intersection $\mf{C}\cap \mf{C}'$ is also an $i$-level Babel sector, then $d_{A}(x,x')\in {}^{i*}\bb{R}$ holds.
\end{rem}

\begin{exa}
In $\Sigma(2,A_{1})$, an affine sector is isometric to the half-line $\bb{R}_{\geq 0}$, while a $2$-level Babel sector is isometric to $\bb{R}_{\geq 0}+\coprod_{n\geq 1}(2n\omega_{2}+\bb{R})$.
\end{exa}

\begin{dfn}
Let $A$ be an $n$-level Babel apartment and $C$ a chamber of $A$. We define the retraction map
$$
\tau_{C}\colon A\longrightarrow \overline{C}
$$
by the condition that $\tau_{C}\vert_{w\overline{C}}$ acts as the inverse of $w\colon \overline{C} \to w\overline{C}$ (i.e., $\tau_{C}(x) = w^{-1}(x)$ for $x \in w\overline{C}$) for any $w\in W_{n}(\Phi)$. A $W_{n}(\Phi)$-isometry $\psi\colon A_{1}\longrightarrow A_{2}$ between Babel apartments is said to be \textit{type-preserving} if $\psi\circ \tau_{C}=\tau_{\psi(C)}\circ \psi$ holds for some chamber $C$ of $A_{1}$. The property of being type-preserving is independent of the choice of the chamber $C$.
\end{dfn}

Next, we consider a higher-dimensional generalization of the uniqueness lemma. In the classical case ($n=1$), let $\Omega$ be a subset of an affine apartment $\Sigma(1,\Phi)$ and $g$ a $W_{a}(\Phi)$-isometry on $\Sigma(1,\Phi)$. It is a known fact that if $g$ fixes $\Omega$ pointwise, then $g$ fixes the convex closure of $\Omega$.

\begin{dfn}
For any $a\in \Phi$ and $k\in \bb{Z}^{n}$, the subset $\alpha_{a,k}=\{v\in \Sigma(n,\Phi) \mid {}^{n*}(a,v)+k\geq 0\}$ is called a \textit{half-space} of $\Sigma(n,\Phi)$. For any subset $\Omega$ of $\Sigma(n,\Phi)$, we denote by $\mathrm{cl}(\Omega)$ the intersection of all half-spaces containing $\Omega$, and call $\mathrm{cl}(\Omega)$ the \textit{enclosure} of $\Omega$.
\end{dfn}

For any non-empty subset $\Omega$ of $\Sigma(n,\Phi)$, $\mathrm{cl}(\Omega)$ can also be characterized as the set of points $z\in \Sigma(n,\Phi)$ such that no wall separates $z$ from $\Omega$.

\begin{lem}\label{cl-fix}
Let $A$ be a Babel apartment and $g\colon A\longrightarrow A$ a type-preserving $W_{n}(\Phi)$-isometry. For any non-empty subset $\Omega$ of $A$, if $g$ fixes $\Omega$ pointwise, then $g$ fixes $\mathrm{cl}(\Omega)$ pointwise.
\end{lem}
\begin{proof}
In the case where $\Phi=A_{1}$, if $\Omega$ is a singleton, it is clear that $\Omega=\mathrm{cl}(\Omega)$. If the cardinality of $\Omega$ is greater than or equal to $2$, the claim follows immediately from Proposition \ref{HypermetricofBA}. 
Next, consider the case where $\Phi\neq A_{1}$. Let $x\in \mathrm{cl}(\Omega)$. Since $g$ maps walls to walls, we have $g(x)\in \mathrm{cl}(\Omega)$. If there are no walls separating $x$ from $g(x)$, then we must have $g(x)=x$. Suppose there exists a wall $H$ that separates $x$ from $g(x)$. In this case, standard arguments imply that $\Omega\subset H$. By induction on $d=\mathrm{rank}(\Phi)$, we obtain the claim.
\end{proof}

This lemma corresponds to \cite[Proposition 2.4.13]{BT}. It should be noted that when $n>1$, the concept of minimal galleries in a Coxeter complex is not applicable.

\begin{exa}
The $(n-1)$-level Babel apartment $\Sigma(n-1,\Phi)$ is enclosure in $\Sigma(n,\Phi)$. Indeed, if $x\in \Sigma(n,\Phi)$ is not contained in $\Sigma(n-1,\Phi)$, then we have $d(0,x)\notin {}^{(n-1)*}\bb{R}$. So there exist a root $a\in \Phi$ and $k_{1},\ldots,k_{n}\in \bb{Z}$ such that $k_{n}\neq 0$ and the wall $H_{a,k}$ separates $x$ and $0$, where $k=k_{n}\omega_{n}+\cdots +k_{1}\omega_{1}$. 
\end{exa}

\begin{exa}
The below figure is an example of enclosed sets. Note that this is not convex unlike the case of $n=1$. Hence it may not be appropriate to call such a set a convex hull. In this example, as can be seen, ${\rm cl}(\{x,y\})$ consists of four affine planes, two affine half-planes and six affine sectors.
\begin{figure}[H]
\begin{center}
\begin{tikzpicture}\label{Figure4}
 \fill[lightgray] (0,0) circle (0.6);
 \fill[lightgray] (2,0) circle (0.6);
 \fill[lightgray] (-1,1.732) circle (0.6);
 \fill[lightgray] (1,1.732) circle (0.6);
 \fill[lightgray] (-2,0) circle (0.6);
 \fill[lightgray] (1,-1.732) circle (0.6);
 \fill[lightgray] (-1,-1.732) circle (0.6);
 \fill[lightgray] (4,0) circle (0.6);
 \fill[lightgray] (3,1.732) circle (0.6);
 \fill[lightgray] (3,-1.732) circle (0.6);
 \fill[lightgray] (5,1.732) circle (0.6);
 \fill[lightgray] (-3,-1.732) circle (0.6);
 \fill[lightgray] (0,-3.464) circle (0.6);
 \fill[lightgray] (-2,-3.464) circle (0.6);
 \fill[lightgray] (-4,-3.464) circle (0.6);
 \fill[lightgray] (2,-3.464) circle (0.6);
 \fill[lightgray] (1,-5.196) circle (0.6);
 \fill[lightgray] (-1,-5.196) circle (0.6);
 \fill[lightgray] (-3,-5.196) circle (0.6);
 \fill[lightgray] (3,-5.196) circle (0.6);
 \coordinate (x) at (-3,-5.196);
 \fill (x) circle (0.06);
 \draw (x) node[left]{$x$};
 \coordinate (y) at (3,1.732);
 \fill (y) circle (0.06);
 \draw (y) node[right]{$y$};
 \draw (x)--(-3,-4.592);
 \draw (-3,-2.332)--(-3,-1.732);
 \draw[dashed] (-3,-4.592)--(-3,-2.332);
 \draw (x)--(-2.480,-4.896);
 \draw (-3,-1.732)--(-2.480,-1.432);
 \draw (0.520,0.300)--(-0.520,-0.300);
 \draw[dashed] (-2.480,-1.432)--(-0.520,-0.300);
 \draw (y)--(2.480,1.432);
 \draw[dashed] (2.480,1.432)--(0.520,0.300);
 \draw (y)--(3,1.132);
 \draw (3,-1.732)--(3,-1.132);
 \draw[dashed] (3,-1.132)--(3,1.132);
 \draw (-0.520,-3.764)--(0.520,-3.164);
 \draw (3,-1.732)--(2.480,-2.032);
 \draw[dashed] (2.480,-2.032)--(0.520,-3.164);
 \draw[dashed] (-0.520,-3.764)--(-2.480,-4.896);
\end{tikzpicture}
\end{center}
\caption{${\rm cl}(\{x,y\})$ in $\Sigma(2,A_{2})$ is inside the parallelogram above.}
\end{figure}
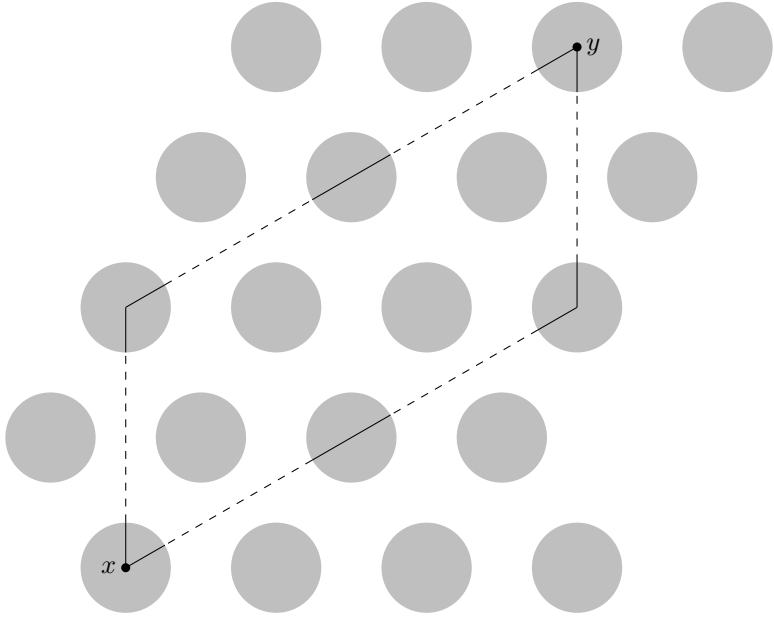
\end{exa}

In addition, we can now define for the $\mb{Bruhat}$ $\mb{order}$, following the $n=1$ case. For any $v,w\in W_{n}(\Phi)$, we define $v\leq w$ to be that $vC_{0}\subset {\rm cl}(C_{0}\cup wC_{0})$. However a study of the Bruhat order on non-Coxeter reflection groups will be future work.

\subsection{The metric structure of Babel buildings and residues at vertices}

\begin{dfn}
Let $X$ be a set and $\mca{A}$ a family of subsets of $X$, each isometric to $\Sigma(n,\Phi)$. For any $i=0,\ldots,n$, a subset $\mf{C}$ of $X$ is called an \textit{$i$-level Babel sector} of $X$ if $\mf{C}$ is an $i$-level Babel sector within some element of $\mca{A}$. We say that the pair $(X, \mca{A})$ is an \textit{$n$-level Babel building} of type $\Phi$ if it satisfies the following conditions:
\begin{enumerate}[(1)]
\item For any $i,j\in \{0,\ldots,n\}$ and any Babel sectors $\mf{C}$ and $\mf{D}$ of levels $i$ and $j$ respectively, there exist a Babel apartment $A\in \mca{A}$ and Babel sub-sectors $\mf{C}'\subset \mf{C}$ and $\mf{D}'\subset \mf{D}$ of the same respective levels such that $\mf{C}'\cup \mf{D}'\subset A$.
\item For any $A_{1},A_{2}\in \mca{A}$, there exists an isometry $\psi\colon A_{1}\longrightarrow A_{2}$ that fixes the intersection $A_{1}\cap A_{2}$ pointwise.
\end{enumerate}
In this case, the family $\mca{A}$ is referred to as a \textit{system of Babel apartments} for $X$.
\end{dfn}

In our framework, a spherical building corresponds to a $0$-level Babel building, and an affine building corresponds to a $1$-level Babel building (cf. \cite[Corollary 2.9.4, 2.9.6]{BT} or \cite[Theorem 11.63]{AB}).

Denote by $\mca{A}_{C}$ the subfamily $\{A\in \mca{A} \mid C\subset A\}$ for any chamber $C$. Similarly, for any $i$-level Babel sector $\mf{C}$, we define
$$
\mca{A}_{\mf{C}}=\{A\in \mca{A} \mid \exists\mf{C}'\subset \mf{C}\text{ such that }\mf{C}'\subset A\},
$$
where $\mf{C}'$ represents a Babel sub-sector of the same level. By definition, the union of all apartments in the system $\mca{A}$ covers $X$.

\begin{prop}
Let $X$ be a Babel building. Then the family $\mca{A}^{\mathrm{comp}}=\{A\subset X \mid A\simeq \Sigma(n,\Phi)\}$ forms a system of Babel apartments.
\end{prop}
\begin{proof}
It suffices to show that for any elements $A_{1},A_{2}$ of $\mca{A}^{\mathrm{comp}}$, there exists an isometry $A_{1}\longrightarrow A_{2}$ that fixes the intersection. If $A_{1}\cap A_{2}$ is empty, there is nothing to show. 

Suppose the intersection is non-empty. Let $F$ be a maximal facet of $A_{1}\cap A_{2}$ and $f$ the dimension of $F$. Let $x_{0},\ldots,x_{f}$ be affinely independent points in $F$. Let us take $x_{f+1},\ldots,x_{d}\in A_{1}$ such that $x_{0},\ldots,x_{d}$ form an affine basis of $A_{1}$. We choose points $y_{f+1},\ldots,y_{d}$ in $A_{2}$ satisfying
$$
d_{A_{2}}(y_{i}, y_{j})=d_{A_{1}}(x_{i},x_{j})
$$
for all $i,j=0,\ldots,d$ (where we set $y_k = x_k$ for $0 \le k \le f$).
Then, by Proposition \ref{HypermetricofBA}, we can define a map $\varphi_{F}\colon A_{1}\longrightarrow A_{2}$ such that
$$
d_{A_{2}}(y_{i},\varphi_{F}(x))=d_{A_{1}}(x_{i},x)
$$
for any $x\in A_{1}$. Since $\dim F = \dim A_{1}\cap A_{2}$, $\varphi_{F}$ is an isometry that fixes the intersection.
\end{proof}

We call $\mca{A}^{{\rm comp}}$ the complete system of Babel apartments of $X$.

\begin{dfn}
Let $X$ be a Babel building, $A\in \mca{A}$ an apartment, and $C$ a chamber of $A$. For any $x\in X$, there exists a Babel apartment $A'\in \mca{A}$ containing both $x$ and $C$. By definition, there exists an isometry $\psi_{A'}\colon A'\longrightarrow A$ fixing the intersection $A\cap A'$ pointwise. We define $\rho_{A,C}(x)=\psi_{A'}(x)$. 

Let us verify that this map is well-defined. Since $C$ is a chamber, we can choose an affine basis $x_{0},\ldots,x_{d}$ within $C$. Suppose $A'$ and $A''$ are Babel apartments of $X$ containing both $x$ and $C$. Then, for any $i=0,\ldots,d$, we have
$$
d_{A}(x_{i},\psi_{A'}(x))=d_{A}(x_{i},\psi_{A''}(x)),
$$
because $\psi_{A'}$ and $\psi_{A''}$ are isometries fixing $C$ pointwise. Hence, by Proposition \ref{HypermetricofBA}, we obtain $\psi_{A'}(x)=\psi_{A''}(x)$. 
The map
$$
\rho=\rho_{A,C}\colon X\longrightarrow A
$$
is called the \textit{retraction} onto $A$ centered at $C$. It is evident that $\rho^{-1}(C)=C$.
\end{dfn}

\begin{rem}
For any $i$-level Babel sector $\mf{C}$ of $X$ and any Babel apartment $A$ containing an $i$-level Babel sub-sector of $\mf{C}$, we can construct the retraction
$$
\rho_{A,\mf{C}}\colon X\longrightarrow A
$$
in a manner analogous to the definition above. However, we will not discuss these general retractions in this paper.
\end{rem}

For any $x,y\in X$, we define the value $d_{X}(x,y)$ to be $d_{A}(x,y)$, where $A$ is any Babel apartment containing both $x$ and $y$. This assignment is well-defined. In what follows, we strictly refer to $d_{X}$ as the \textit{distance function} on $X$.

\begin{thm}
For any Babel apartment $A$ of $X$ and any chamber $C$ of $A$, the retraction $\rho=\rho_{A,C}\colon X\longrightarrow A$ is distance-decreasing; that is, the inequality
$$
d_{A}(\rho(x),\rho(y))\leq d_{X}(x,y)
$$
holds for any $x,y\in X$.
\end{thm}
\begin{proof}
First, suppose that $x,y$ are vertices of $X$. Let $A'$ be a Babel apartment which contains $x,y$. We show that the geodesic defined by $[x,y]=\{z\in A' \mid d_{A'}(x,y)=d_{A'}(x,z)+d_{A'}(z,y)\}$ is covered by finitely many elements of $\mca{A}_{C}$. Suppose there exist $(n-1)$-level Babel sectors $\mf{C}_{1}, \ldots, \mf{C}_{m}$ such that $[x,y] \subset \mf{C}_{1}\cup \cdots \cup \mf{C}_{m}$. Then, for each $i=1,\ldots, m$, we can choose an $(n-1)$-level Babel sub-sector $\mf{C}_{i}' \subset \mf{C}_{i}$ such that $\mf{C}_{i}'$ is included in some apartment belonging to $\mca{A}_{C}$. The geodesic connecting the cone points of $\mf{C}_{i}'$ and $\mf{C}_{i}$ is covered by finitely many $(n-2)$-level Babel sectors. Consequently, we have a decomposition
$$
\mf{C}_{i} = \mf{C}_{i}'\cup\mf{D}_{1}\cup\cdots\cup\mf{D}_{k},
$$
where each $\mf{D}_{j}$ is an $(n-2)$-level Babel sector. Similarly, for each $j$, we can choose an $(n-2)$-level Babel sub-sector $\mf{D}_{j}'\subset \mf{D}_{j}$ such that $\mf{D}_{j}'$ is included in some apartment belonging to $\mca{A}_{C}$. By repeating this argument via descending induction down to level $0$, we obtain the conclusion.

Based on the discussion above, we assume that $x$ and $y$ are contained in adjacent $m$-level Babel sectors $\mf{C}$ and $\mf{D}$ of two Babel apartments $A_{1}, A_{2} \in \mca{A}_{C}$, respectively. Note that the three objects---the chamber $C$, $\mf{C}$, and $\mf{D}$---may not necessarily lie within the same Babel apartment. 

\textbf{Case 1.} First, we suppose that $\mf{C}$ and $\mf{D}$ are included in $A'$.
Let $\varphi_{i}$ denote the restriction of the retraction $\rho_{A,C}\colon X \longrightarrow A$ to $A_{i}$ for $i=1,2$.
There exist elements $w_{c},w_{d} \in W(\Phi)\ltimes \bigoplus_{j>m}\bb{Z}\omega_{j}(\Phi^{\vee})$ such that 
$$
\mf{C}=w_{c}(\mathbf{D}_{m}), \quad \mf{D}=w_{d}(\mathbf{D}_{m}),
$$
where $\mathbf{D}_{m}$ denotes a fundamental domain of the group $W(\Phi)\ltimes \bigoplus_{j>m}\bb{Z}\omega_{j}(\Phi^{\vee})$.
For $i=1,2$, let $E_{i}$ be the hyper-affine enveloping space of $A_{i}$, and let $\tilde{\varphi}_{i}\colon E_{i}\longrightarrow E_{A}$ denote the extension of the isometry $\varphi_{i}$.
Let $\tilde{\mathbf{D}}_{m}$ denote the fundamental domain of $W(\Phi)\ltimes \bigoplus_{j>m}\bb{Z}\omega_{j}(\Phi^{\vee})$ in the enveloping space, and define
$$
\tilde{\mf{C}}=w_{c}(\tilde{\mathbf{D}}_{m}), \quad \tilde{\mf{D}}=w_{d}(\tilde{\mathbf{D}}_{m}).
$$
Consequently, the lifted sectors $\tilde{\mf{C}}$ and $\tilde{\mf{D}}$ are adjacent in the enveloping space. Let $z$ be a point in the intersection $[x,y]_{E}\cap \overline{(\tilde{\mf{C}}\cap \tilde{\mf{D}})}$, where $[x,y]_{E}$ denotes the geodesic segment connecting $x$ and $y$ within the enveloping space.
Since $z$ lies in the intersection, we have $\tilde{\varphi}_{1}(z)=\tilde{\varphi}_{2}(z)$.
Thus, applying the triangle inequality in $E_{A}$, we obtain:
\begin{align*}
d_{A}(\rho(x),\rho(y)) &\leq d_{E_{A}}(\tilde{\varphi}_{1}(x),\tilde{\varphi}_{1}(z)) + d_{E_{A}}(\tilde{\varphi}_{2}(z),\tilde{\varphi}_{2}(y)) \\
&=d_{E_{1}}(x,z)+d_{E_{2}}(z,y) \\
&=d_{E_{A'}}(x,z)+d_{E_{A'}}(z,y) \\
&=d_{X}(x,y).
\end{align*}

\textbf{Case 2.} Next, we consider the general case where $\mf{C}$ and $\mf{D}$ are not necessarily included in $A'$. Let $\mf{C}'$ and $\mf{D}'$ be sub-sectors of $\mf{C}$ and $\mf{D}$ of the same level, respectively, such that $\mf{C}', \mf{D}'\subset A'$. Then there exist Babel sectors $\mf{E}_{1},\ldots, \mf{E}_{p}$ and $\mf{F}_{1}, \ldots, \mf{F}_{q}$ of level strictly less than $m$ such that
\begin{align*}
\mf{C}&=\mf{C}'\cup\mf{E}_{1}\cup\cdots \cup \mf{E}_{p}, \\
\mf{D}&=\mf{D}'\cup\mf{F}_{1}\cup\cdots \cup \mf{F}_{q},
\end{align*}
and for any $r=1,\ldots, p$ and $s=1,\ldots, q$, the sectors $\mf{E}_{r}$ and $\mf{F}_{s}$ are included in some apartments belonging to $\mca{A}_{C}$.

Let $x', y'$ be the cone points of $\mf{C}', \mf{D}'$ respectively. 
Note that the paths connecting $x$ to $x'$ and $y$ to $y'$ are covered by sectors of level less than $m$. Therefore, by the induction hypothesis on the level, the retraction $\rho$ is distance-decreasing on these paths.
Thus we have
\begin{align*}
d_{A}(\rho(x), \rho(y)) 
&\leq d_{A}(\rho(x), \rho(x')) + d_{A}(\rho(x'),\rho(y')) + d_{A}(\rho(y'),\rho(y)) \\
&\leq d_{X}(x,x')+d_{X}(x',y')+d_{X}(y',y) \\
&=d_{X}(x,y).
\end{align*}
This completes the proof.
\end{proof}

The proof above relies on the logic established for the $n=1$ case (i.e., affine buildings). However, for $n>1$, a geodesic segment in $X$ is not necessarily partitioned by finitely many chambers.

\begin{cor}
Let $X$ be an $n$-level Babel building. Then $(X,d_{X})$ constitutes an ${}^{n*}\bb{R}$-metric space. Moreover, for any $x,y,z\in X$, if $t\in {}^{n*}[0,1]$ is such that the point $p_{t}=(1-t)x+ty$ lies in $X$, then the CAT($0$)-inequality
$$
d_{X}^{2}(p_{t},z)\leq (1-t)d_{X}^{2}(x,z)+td_{X}^{2}(y,z)-t(1-t)d_{X}^{2}(x,y)
$$
holds.
\end{cor}
\begin{proof}
Let $x,y,z\in X$. Let $A$ be a Babel apartment containing $x$ and $y$. For a chamber $C$ of $A$, consider the retraction $\rho=\rho_{A,C}\colon X\longrightarrow A$. Since $x,y\in A$ and $\rho$ is distance-decreasing, we have
$$
d_{X}(x,y) = d_{A}(\rho(x),\rho(y)) \leq d_{A}(\rho(x),\rho(z))+d_{A}(\rho(z),\rho(y)) \leq d_{X}(x,z)+d_{X}(y,z).
$$
Hence $(X,d_{X})$ is a hyper-metric space. 
Moreover, if $p_{t}$ belongs to $X$, by choosing a chamber $C$ such that $p_{t}\in \overline{C}$, we have $\rho(p_{t})=p_{t}$.
Assuming the equality $d_{X}(p_{t},\rho(z))=d_{X}(p_{t},z)$ holds for a suitable choice of $C$, combined with the inequalities $d_{X}(x,\rho(z))\leq d_{X}(x,z)$ and $d_{X}(y,\rho(z))\leq d_{X}(y,z)$, we can apply Proposition \ref{HypermetricofBA} in $A$ to obtain:
$$
d_{X}^{2}(p_{t},z)\leq (1-t)d_{X}^{2}(x,z)+td_{X}^{2}(y,z)-t(1-t)d_{X}^{2}(x,y).
$$
\end{proof}

As a remark, when $n\geq 2$, an $n$-level Babel apartment is not convex. So an $n$-level Babel building is not convex also. Hence it is not necessarily $p_{t}=(1-t)x+ty\in X$ for any $x,y\in X$ and $t\in {}^{n*}[0,1]$.

\begin{rem}
We can now provide a brief description of the geometric shape of $n$-level Babel buildings. For each $i=1,\ldots,n$, let $X_{i}$ be an affine building scaled by the factor $\omega_{i}$. Specifically, the metric on $X_{i}$ takes values in ${}^{i*}\bb{R}$:
$$
d_{X_{i}}\colon X_{i}\times X_{i}\longrightarrow {}^{i*}\bb{R}.
$$
Let $\mca{V}(X_{i})$ denote the set of vertices of $X_{i}$. The construction proceeds recursively: start with the vertices $\mca{V}(X_{n})$. We replace each vertex in $\mca{V}(X_{n})$ with a copy of the space $X_{n-1}$ (or its vertex set). Repeat this operation recursively. Finally, the space obtained by replacing the vertices at the last stage with $X_{1}$ constitutes the Babel building. 
Thus, a Babel building forms a nested structure of affine buildings. This structure parallels the fact that $n$-dimensional local fields are obtained as nested structures of complete discrete valuation fields. In the preface of their monumental paper \cite{BT}, Bruhat and Tits explained their theory as follows:
\begin{quote}
De fa\c{c}on imag\'ee, la th\'eorie qui sera expos\'ee ici, \`a partir du chapitre II, peut \^etre d\'ecrite comme l'\'etude d'un groupe semi-simple d\'efini sur un corps local $K$ — ou plut\^ot du groupe $G$ de ses points rationnels sur $K$ — consid\'er\'e comme \guillemotleft{} objet de dimension infinie \guillemotright{} d\'efini sur le corps r\'esiduel $k$.
\end{quote}
This philosophy is shared by our theory. Just as a $1$-level Babel building (an affine building) can be viewed as an infinite collection of $0$-level Babel buildings (spherical buildings), an $n$-level Babel building is an infinite collection of $(n-1)$-level Babel buildings.
\end{rem}

\begin{dfn}
Let $(Y,d_{Y})$ be an ${}^{n*}\bb{R}$-metric space and let $B$ be a subset of $Y$.
\begin{enumerate}[(i)]
\item A subset $B$ is said to be \textit{bounded} if for any $y\in Y$, there exists a supremum
$$
r(y,B)=\sup_{b\in B} d_{Y}(y,b).
$$
\item If the infimum
$$
r(B)=\inf_{y\in Y}r(y,B)
$$
exists, then $r(B)$ is called the \textit{circumradius} of $B$. A point $y\in Y$ satisfying $r(B)=r(y,B)$ is called a \textit{circumcenter} of $B$.
\end{enumerate}
\end{dfn}

\begin{lem}
Let $X$ be an $n$-level Babel building. If $B$ is a bounded subset of $X$, then there exists at most one circumcenter of $B$.
\end{lem}
\begin{proof}
Suppose that $x,y\in X$ are circumcenters of $B$. Assume there exists $t\in {}^{n*}(0,1)$ such that the point $p_{t}=(1-t)x+ty$ belongs to $X$. By applying the CAT($0$)-inequality, we obtain
$$
r^{2}(p_{t},B)\leq (1-t)r^{2}(x,B)+tr^{2}(y,B)-t(1-t)d_{X}^{2}(x,y).
$$
Since $x$ and $y$ are circumcenters, we have $r(x,B)=r(y,B)=r(B)$. Furthermore, by the definition of the circumradius, $r(B) \leq r(p_{t},B)$ holds. Consequently,
$$
r^{2}(B) \leq r^{2}(p_{t},B) \leq r^{2}(B) - t(1-t)d_{X}^{2}(x,y).
$$
This implies
$$
t(1-t)d_{X}^{2}(x,y) \leq 0.
$$
Since $t(1-t) > 0$, we must have $d_{X}(x,y)=0$, which implies $x=y$.
\end{proof}

We now state, as a proposition, the claim of a higher-dimensional generalization of the Bruhat-Tits fixed point theorem.

\begin{prop}
Let $G$ be a group acting isometrically on an $n$-level Babel building $X$, and let $B$ be a bounded subset of $X$. If $G$ stabilizes $B$ and $B$ admits a circumcenter, then the circumcenter of $B$ is a fixed point of the $G$-action on $X$.
\end{prop}
\begin{proof}
This follows immediately from the previous lemma. Since $G$ stabilizes $B$ and acts by isometries, it must map a circumcenter to a circumcenter. By the uniqueness of the circumcenter, the point must be fixed.
\end{proof}

In contrast to the classical case where $n=1$, the space $X$ is not convex; consequently, a bounded subset does not necessarily possess a circumcenter. Therefore, our fixed point theorem requires an explicit assumption regarding the existence of a circumcenter, a condition that is automatically satisfied at level $1$.

\begin{dfn}
Let $X$ be an $n$-level Babel building and $p$ a vertex of $X$. For any subset $U$ of $X$, we define
$$
r_{p}U=\{u\in U \mid d_{X}(p,u)\in {}^{(n-1)*}\bb{R}\}.
$$
We call this set the \textit{residue} of $U$ at the vertex $p$.
\end{dfn}

\begin{thm}
For any vertex $p$ of $X$, the residue
$$
r_{p}X=\{x\in X \mid d_{X}(p,x)\in {}^{(n-1)*}\bb{R}\}
$$
constitutes an $(n-1)$-level Babel building.
\end{thm}
\begin{proof}
Let $\mca{A}$ be a system of Babel apartments of $X$. We define the family of subsets
$$
r_{p}\mca{A}=\{r_{p}A \mid A\in \mca{A}\}\setminus\{\emptyset\}.
$$
We claim that $r_{p}\mca{A}$ forms a system of Babel apartments for $r_{p}X$. 
First, note that any element of $r_{p}\mca{A}$ is an $(n-1)$-level Babel apartment (by the nesting structure of apartments). 

Let $\mf{C}$ and $\mf{D}$ be Babel sectors of levels $i$ and $j$ in $r_{p}X$, respectively. Note that these are also Babel sectors of $X$. By condition (1) of the definition of a Babel building, there exists a Babel apartment $A \in \mca{A}$ containing some Babel sub-sectors $\mf{C}'\subset \mf{C}$ and $\mf{D}'\subset \mf{D}$ of the corresponding levels. 
Let $x, x'$ and $y, y'$ be the cone points of $\mf{C}, \mf{C}'$ and $\mf{D}, \mf{D}'$, respectively. 
By the property of sub-sectors in a Babel apartment, we have
$$
d_{X}(x,x'), \ d_{X}(y,y')\in {}^{(n-1)*}\bb{R}.
$$
Since $\mf{C},\mf{D}\subset r_{p}X$, we know that
$$
d_{X}(p,x), \ d_{X}(p,y)\in {}^{(n-1)*}\bb{R}.
$$
By the triangle inequality, we obtain
$$
d_{X}(p,x'), \ d_{X}(p,y')\in {}^{(n-1)*}\bb{R}.
$$
Consequently, $\mf{C}',\mf{D}'\subset r_{p}A$. 

Next, for any two elements $r_{p}A_{1}$ and $r_{p}A_{2}$ of $r_{p}\mca{A}$, we construct an isometry $r_{p}\psi\colon r_{p}A_{1}\longrightarrow r_{p}A_{2}$ that fixes the intersection pointwise. Suppose that the intersection $r_{p}A_{1}\cap r_{p}A_{2}$ is non-empty. Let $\psi\colon A_{1}\longrightarrow A_{2}$ be an isometry that fixes the intersection $A_{1}\cap A_{2}$. We define $r_{p}\psi$ to be the restriction of $\psi$ to $r_{p}A_{1}$. Since $\psi$ preserves distances and fixes points in the intersection (which are close to $p$), it maps $r_{p}A_{1}$ onto $r_{p}A_{2}$.
Therefore, $r_{p}X$ is a Babel building of level $n-1$, and $r_{p}\mca{A}$ serves as its system of Babel apartments.
\end{proof}

\begin{dfn}
We refer to the pair $(r_{p}X,r_{p}\mca{A})$ as the \textit{residue Babel building} of $(X,\mca{A})$ at a vertex $p\in X$.
\end{dfn}

\begin{cor}
For any vertices $p,q\in X$, if $d_{X}(p,q)\in {}^{(n-1)*}\bb{R}$, then $r_{p}X=r_{q}X$. Otherwise, the intersection $r_{p}X\cap r_{q}X$ is empty.
\end{cor}

\begin{cor}
For any vertex $p\in X$, the set $r_{p}^{n-1}X=\{x\in X \mid d_{X}(p,x)\in \bb{R}\}$ constitutes an affine building.
\end{cor}

By construction, the retraction
$$
\rho_{r_{p}A,C}\colon r_{p}X\longrightarrow r_{p}A
$$
coincides with the restriction of $\rho_{A,C}$ to the residue $r_{p}X$, provided that $C$ is a chamber of $A$ contained in $r_{p}X$ (i.e., $r_{p}C=C$).

For any vertices $p,q\in X$, we write $p\sim q$ if $d_{X}(p,q)\in \bb{R}$. Let $\mca{B}(X)$ denote the quotient of the set of vertices by this equivalence relation. Decomposing $X$ into connected components with respect to this relation, we obtain
$$
X=\coprod_{[p]\in \mca{B}(X)} r_{p}^{n-1}X.
$$
Let $G$ be a group acting isometrically on $X$, and let $\Gamma$ be a subgroup of $G$. Define $\Gamma_{p}=\{\gamma\in \Gamma \mid \gamma(r_{p}^{n-1}X)=r_{p}^{n-1}X\}$ for any $[p]\in \mca{B}(X)$. Then we have the decomposition
$$
\Gamma\backslash X=\coprod_{\Gamma[p]\in \Gamma\backslash \mca{B}(X)} \Gamma_{p}\backslash r_{p}^{n-1}X,
$$
and the homology isomorphism
$$
H_{*}(\Gamma\backslash X,\bb{Z})\simeq \bigoplus_{\Gamma[p]\in \Gamma\backslash \mca{B}(X)} H_{*}(\Gamma_{p}\backslash r_{p}^{n-1}X,\bb{Z}).
$$

In future work, we intend to investigate the homology and cohomology of certain arithmetic quotients of Babel buildings.

\section{Group actions on Babel buildings}

Let $\widehat{G}$ be a group acting isometrically on an $n$-level Babel building $X$. Denote by $\mca{AC}$ the set of pairs $(A,C)$ consisting of a Babel apartment $A$ of $X$ and a chamber $C$ of $A$. We say that the action of $\widehat{G}$ on $X$ is \textit{strongly transitive} if $\widehat{G}$ acts transitively on the set $\mca{AC}$. In this section, we assume this condition. Let $G$ be a subgroup of $\widehat{G}$ whose action on $X$ is type-preserving.

\subsection{Fixers and stabilizers}

Let $\Omega$ be a subset of $X$. The fixer and the stabilizer of $\Omega$ in $\widehat{G}$ are denoted by $\widehat{P}_{\Omega}$ and $\widehat{P}_{\Omega}^{\dag}$, respectively.
We define $P_{\Omega}=G\cap \widehat{P}_{\Omega}$ and $P_{\Omega}^{\dag}=G\cap \widehat{P}_{\Omega}^{\dag}$. 

Let us fix a pair $(A,C) \in \mca{AC}$. We define
$$
\widehat{N}=\widehat{P}_{A}^{\dag}, \quad \widehat{B}=\widehat{P}_{C}.
$$
It is clear that $g\widehat{P}_{\Omega}g^{-1}=\widehat{P}_{g(\Omega)}$ for any $g\in \widehat{G}$. If $\Omega$ is a subset of $A$, then 
$$
P_{\Omega}=P_{\mathrm{cl}(\Omega)}.
$$
Moreover, if $\Omega$ contains an affine basis of $A$, then
$$
\widehat{P}_{\Omega}=\widehat{P}_{\mathrm{cl}(\Omega)}.
$$
Since the chamber $C$ contains an affine basis of $A$, we have $\widehat{B}\cap \widehat{N}=\widehat{P}_{A}$ and $B\cap N=P_{A}$. Let
\begin{align*}
\widehat{\nu}&\colon \widehat{N}\longrightarrow \widehat{W}=\widehat{N}/(\widehat{B}\cap \widehat{N}),\\
\nu &\colon N\longrightarrow W=N/(B\cap N)
\end{align*}
be the canonical epimorphisms, and put $\widehat{H}=\widehat{B}\cap \widehat{N}$ and $H=B\cap N$.

Let $\mathbf{D}_{i}$ be a fundamental domain in $A$ as defined in Definition \ref{midvecCham}. Let $d$ be an arbitrary point in $\mathbf{D}_{i}$ and define
$$
E_{\mathbf{D}_{i}}=\{x+\mathbf{D}_{i} \mid d_{X}(d,x)\in {}^{i*}\bb{R}\}.
$$
Since $\widehat{P}_{\mf{C}_{1}}\cup \widehat{P}_{\mf{C}_{2}}\subset \widehat{P}_{\mf{C}_{1}\cap \mf{C}_{2}}$ holds for any $\mf{C}_{1},\mf{C}_{2}\in E_{\mathbf{D}_{i}}$, the union
$$
\widehat{\mf{B}}^{0}_{\mathbf{D}_{i}}=\bigcup_{\mf{C}\in E_{\mathbf{D}_{i}}}\widehat{P}_{\mf{C}}
$$
forms a group. Let $\widehat{V}_{i}$ be the subgroup of $\widehat{W}$ consisting of translations by vectors in ${}^{i*}\bb{R}$. Then the inverse image $\widehat{\nu}^{-1}(\widehat{V}_{i})$ normalizes $\widehat{\mf{B}}^{0}_{\mathbf{D}_{i}}$.

\begin{lem}\label{BT2.5.8}
For any two Babel apartments $A_{1}$ and $A_{2}$, there exists an element $g\in G$ such that $gA_{1}=A_{2}$ and $g(x)=x$ for all $x\in A_{1}\cap A_{2}$. 
\end{lem}
\begin{proof}
Suppose that the intersection $A_{1}\cap A_{2}$ is non-empty. We can choose a chamber $D\subset A_{1}\cap A_{2}$. By the strong transitivity of the action of $G$, there exists $g\in G$ such that $gA_{1}=A_{2}$ and $gD=D$. Since the action of $G$ on $X$ is type-preserving and fixes the chamber $D$ set-wise, we have $g\vert_{D}=\mathrm{id}_{D}$. Furthermore, since $D$ contains an affine basis of $A_{1}$, the isometry $g$ must fix the entire intersection $A_{1}\cap A_{2}$ pointwise (as $A_{1}\cap A_{2}$ is contained in the affine subspace spanned by $D$).
\end{proof}

\begin{lem}
Let $\Omega$ be a subset of $A$. Define $\widehat{N}_{\Omega}^{\dag}=\widehat{P}_{\Omega}^{\dag}\cap \widehat{N}$ and $\widehat{N}_{\Omega}=\widehat{P}_{\Omega}\cap \widehat{N}$. Then the equalities
$$
\widehat{P}_{\Omega}^{\dag}=\widehat{N}_{\Omega}^{\dag}.P_{\Omega}, \quad \widehat{P}_{\Omega}=\widehat{N}_{\Omega}.P_{\Omega}
$$
hold.
\end{lem}
\begin{proof}
Since $P_{\Omega}$ acts transitively on the set of Babel apartments containing $\Omega$, for any $g\in \widehat{P}_{\Omega}^{\dag}$, there exists $g'\in P_{\Omega}$ such that $gA=g'A$. Consequently, we have $g^{-1}g'\in \widehat{N}_{\Omega}^{\dag}$ by Lemma \ref{BT2.5.8} (applied to $A$ and $g^{-1}g'A=A$, fixing $\Omega$).
This implies $g \in P_{\Omega}\widehat{N}_{\Omega}^{\dag} = \widehat{N}_{\Omega}^{\dag}P_{\Omega}$. Thus, we obtain $\widehat{P}_{\Omega}^{\dag}=\widehat{N}_{\Omega}^{\dag}.P_{\Omega}$. Similarly, we can prove the equality $\widehat{P}_{\Omega}=\widehat{N}_{\Omega}.P_{\Omega}$.
\end{proof}

\subsection{Double cosets}

Let $Q$ and $Q'$ be subgroups of $\widehat{G}$ containing $\widehat{H}=\widehat{B}\cap \widehat{N}$, and define
$$
\widehat{W}_{Q}=\widehat{\nu}(\widehat{N}\cap Q), \quad \widehat{W}_{Q'}=\widehat{\nu}(\widehat{N}\cap Q').
$$
For any $w\in \widehat{W}$, the set $(\widehat{N}\cap Q)n(\widehat{N}\cap Q')$ corresponds to the pre-image of the double coset $\widehat{W}_{Q}w\widehat{W}_{Q'}$, where $n\in \widehat{N}$ is any representative such that $\widehat{\nu}(n)=w$. 
Let $\lambda_{Q,Q'}$ denote the canonical epimorphism 
$$
Q\backslash Q\widehat{N}Q'/Q'\longrightarrow \widehat{W}_{Q}\backslash \widehat{W}/\widehat{W}_{Q'}.
$$ 
The following proposition is an analogue of \cite[Proposition 4.2.1]{BT} within our theory.

\begin{prop}\label{4.2.1}
Let $\Omega,\Omega'$ be subsets of $A$. Let $Q,Q'$ be subgroups of $\widehat{G}$ satisfying the inclusions:
$$
\widehat{H}P_{\Omega}\subset Q\subset \widehat{P}_{\Omega}^{\dag}, \quad \widehat{H}P_{\Omega'}\subset Q'\subset \widehat{P}_{\Omega'}^{\dag}.
$$
Then the map $\lambda_{Q,Q'}$ is bijective.
\end{prop}
\begin{proof}
We prove the injectivity of $\lambda_{Q,Q'}$. 
Let $n,n'\in \widehat{N}$. Suppose that there exist $q\in Q$ and $q'\in Q'$ satisfying the equation $qn=n'q'$. 
Since $q\in Q\subset \widehat{P}_{\Omega}^{\dag}$, we have $q(\Omega)=\Omega$, which implies $\Omega\subset A\cap q(A)$. 
Also, since $q'\in Q'\subset \widehat{P}_{\Omega'}^{\dag}$, we have $qn(\Omega')=n'q'(\Omega')=n'(\Omega')$. 
By Lemma \ref{BT2.5.8}, there exists $g\in G$ such that $q(A)=g(A)$ and the restriction $g\vert_{A\cap q(A)}$ is the identity. 
Since $\Omega \subset A\cap q(A)$, $g$ fixes $\Omega$ pointwise, so $g\in P_{\Omega}$. 
Furthermore, since $n'(\Omega')\subset A$ and $n'(\Omega')=qn(\Omega')\subset q(A)$, the set $n'(\Omega')$ lies in the intersection $A\cap q(A)$. Thus $g$ fixes $n'(\Omega')$ pointwise, which implies $g\in P_{n'(\Omega')}$.

By the assumptions $\widehat{H}P_{\Omega}\subset Q$ and $\widehat{H}P_{\Omega'}\subset Q'$, we have $g\in P_{\Omega}\subset Q$ and $g\in P_{n'(\Omega')}=n'P_{\Omega'}n'^{-1}\subset n'Q'n'^{-1}$. 
Let us define $m=g^{-1}qn$. Note that $g^{-1}q(A)=A$, so $g^{-1}q \in \widehat{N}$, which implies $m \in \widehat{N}$.
Since $g\in Q$, we have $m \in (Q\cap \widehat{N})n$. 
On the other hand, since $g\in n'Q'n'^{-1}$, we can write $m = g^{-1}n'q' = (n'y^{-1}n'^{-1})n'q' = n'(y^{-1}q')$ for some $y\in Q'$. Since $m, n' \in \widehat{N}$, we must have $y^{-1}q' \in \widehat{N}\cap Q'$.
Consequently, $n'\in (\widehat{N}\cap Q)n(\widehat{N}\cap Q')$.
This implies that the map $\lambda_{Q,Q'}$ is injective.
\end{proof}

\begin{prop}\label{BD}
The Bruhat decomposition
$$
\widehat{G}=\coprod_{w\in \widehat{W}}\widehat{B}w\widehat{B}
$$
holds.
\end{prop}
\begin{proof}
Let $g\in \widehat{G}$ be any element. By the definition of a Babel building, there exists a Babel apartment $A'$ containing both $C$ and $gC$. Since the action of $\widehat{B}$ on the set $\mca{A}_{C}$ is transitive, there exists an element $b\in \widehat{B}$ such that $bA=A'$. 
Consequently, $b^{-1}gC$ is a chamber of $A$ (since $gC \subset A'=bA$). 
Since $\widehat{N}$ acts transitively on the chambers of $A$, there exists $n\in \widehat{N}$ such that $b^{-1}gC=nC$. 
This implies that $n^{-1}b^{-1}g$ fixes $C$, so $n^{-1}b^{-1}g \in \widehat{B}$. Thus, $g \in \widehat{B}n\widehat{B}$.
Therefore, we have $\widehat{G}=\widehat{B}\widehat{N}\widehat{B}$. 
The disjointness of the decomposition follows directly from Proposition \ref{4.2.1} (by setting $Q=Q'=\widehat{B}$ and noting that $\widehat{W}_{\widehat{B}}=\{1\}$).
\end{proof}

Let $o$ be a vertex of the chamber $C$, and let $K=\widehat{P}_{o}$ denote its fixer. Note that $K$ admits the decomposition $K = \widehat{B}W(\Phi)\widehat{B}$. Let $\mathbf{D}$ be a fundamental domain for the action of $W(\Phi)$ on $A$. Let $\widehat{V}_{\mathbf{D}}$ denote the subset of $\widehat{W}$ consisting of translations by vectors in $\mathbf{D}$.

\begin{prop}\label{CD}
The Cartan decomposition
$$
\widehat{G}=\coprod_{v\in \widehat{V}_{\mathbf{D}}}KvK
$$
holds.
\end{prop}
\begin{proof}
It suffices to show that $\widehat{G}=K\widehat{N}K$. Since $\widehat{B}\subset K$, this follows from the Bruhat decomposition.
By Proposition \ref{4.2.1} applied to $Q=Q'=K$, the double cosets are in one-to-one correspondence with the quotient space $\widehat{W}_{K}\backslash \widehat{W}/\widehat{W}_{K}$.
Since $\widehat{W}_{K} \cong W(\Phi)$, this quotient is represented by $\widehat{V}_{\mathbf{D}}$.
Thus, the statement holds.
\end{proof}

\begin{prop}\label{4.3.8}
Let $x,y,z,u$ be four points in $A$ such that $y,z\in [x,u]$ with $d_{A}(x,y)\leq d_{A}(x,z)$. Then the equality
$$
\widehat{P}_{x}\widehat{P}_{u}\cap \widehat{P}_{y}\widehat{P}_{z}=(\widehat{P}_{x}\cap \widehat{P}_{y})(\widehat{P}_{z}\cap \widehat{P}_{u})
$$
holds.
\end{prop}
\begin{proof}
Let us take elements $p\in \widehat{P}_{x}$ and $p'\in \widehat{P}_{u}$ such that $pp'\in \widehat{P}_{y}\widehat{P}_{z}$. We aim to show that $p\in \widehat{P}_{y}$ and $p'\in \widehat{P}_{z}$. 
Let $g=pp'$. By assumption, we can write $g=qq'$ with $q\in \widehat{P}_{y}$ and $q'\in \widehat{P}_{z}$. Define $z'=g(z)$ and $u'=g(u)$. 
Note that since $q'$ fixes $z$ and $q$ fixes $y$, we have $d(y, z') = d(q(y), q(q'(z))) = d(y, z)$.
Similarly, since $p$ fixes $x$ and $p'$ fixes $u$, we have $u'=p(u)$.
We establish the following inequalities:
\begin{align}
d(y,u') &\leq d(y,z') + d(z',u') = d(y,z) + d(g(z), g(u)) \notag \\
        &= d(y,z) + d(z,u) = d(y,u), \label{ineq(1)} \\ 
d(x,u') &\leq d(x,y) + d(y,u') \leq d(x,y) + d(y,u) \notag \\
        &= d(x,u) = d(p(x), p(u)) = d(x,u'), \label{ineq(2)} \\
d(x,z') &\leq d(x,y) + d(y,z') = d(x,y) + d(y,z) \notag \\
        &= d(x,z), \label{ineq(3)} \\ 
d(x,u') &\leq d(x,z') + d(z',u') \leq d(x,z) + d(z,u) \notag \\
        &= d(x,u) = d(x,u'). \label{ineq(4)}
\end{align}
From the inequalities in \eqref{ineq(2)}, all intermediate inequalities must be equalities. Thus $d(x,y)+d(y,u')=d(x,u')$, which implies that $y$ lies on the geodesic $[x,u']=[p(x),p(u)]$. Combined with \eqref{ineq(1)}, which implies $d(y, u') = d(y, u)$, we conclude that $y=p(y)$. Thus $p \in \widehat{P}_{y}$.
Similarly, using \eqref{ineq(3)} and \eqref{ineq(4)}, we deduce that $z'$ lies on $[x,u']$ and $d(x, z')=d(x, z)$. Since $z'=p(p'(z))$ and $p(z)=z$ (as $p\in \widehat{P}_{y}\subset \widehat{P}_{z}$ is not yet guaranteed, but $z$ is on $[x,u']$ and $u'=p(u)$...),
\textit{Correction of the logic flow for $p'$:}
From $p \in \widehat{P}_{y}$ and the ordering $x,y,z,u$, $p$ must fix the segment $[x,y]$. Since $g=pp'$, we have $p'=p^{-1}g$.
Since $y=p(y)$, we have $p \in \widehat{P}_x \cap \widehat{P}_y$.
Now consider $p'$. Since $pp'=qq'$, we have $p' = p^{-1}qq'$.
Using the symmetry or applying the same metric argument to $z$ and $p'$, we deduce $p' \in \widehat{P}_z$.
Therefore, we obtain $p\in \widehat{P}_{x}\cap \widehat{P}_{y}$ and $p'\in \widehat{P}_{z}\cap \widehat{P}_{u}$.
\end{proof}

This statement is a higher-dimensional analogue of \cite[Proposition 4.3.8]{BT}. Note, however, that the geodesic $[x,y]$ in $A$ is not necessarily contractible.

\begin{cor}
Under the same assumption as the above statement, let us take $y,z\in A, t\in \widehat{\nu}^{-1}(\widehat{V})$ such that $t(z)\in y+\mb{D}$. Then we have
$$
\widehat{\mf{B}}_{-\mb{D}}^{0}t\widehat{\mf{B}}_{\mb{D}}^{0}\cap \widehat{P}_{y}t\widehat{P}_{z}=\widehat{P}_{y-\mb{D}}t\widehat{P}_{z+\mb{D}}
$$
\end{cor}
\begin{proof}
By multiplication $t^{-1}$ from the right, we assume that $t=1$. If $pp'\in \widehat{P}_{y}\widehat{P}_{z}$ with $p\in \widehat{\mf{B}}_{-\mb{D}}^{0}$ and $p'\in \widehat{\mf{B}}_{\mb{D}}^{0}$, there exist sectors $\mf{C},\mf{C}'$ of direction $-\mb{D},\mb{D}$ respectively and  points $x\in \mf{C}$ and $u\in \mf{C}'$ such that $x,y,z,u$ satisfy the hypothesis of Proposition \ref{4.3.8}. Then we have
$$
p\in \widehat{P}_{y}\cap \widehat{P}_{\mf{C}}=\widehat{P}_{y-\mb{D}}, p'\in \widehat{P}_{z}\cap \widehat{P}_{\mf{C}'}=\widehat{P}_{z+\mb{D}}.
$$
\end{proof}

\begin{exa}
Let $F=F_{2}$ be a $2$-dimensional local field and $t_{1},t_{2}$ a system of local parameters of $F$. Let
$$
v_{F}\colon F\longrightarrow \bb{Z}^{2}\cup \{\infty\}
$$
be the valuation of rank $2$ associated to $t_{1},t_{2}$, that is, we define $v_{F}(t_{1})=\omega_{1}$ and $v_{F}(t_{2})=\omega_{2}$. We denote by
\begin{align*}
O_{F}&=\{x\in F\mathrel{\vert} v_{F}(x)\geq 0\}, \\
\mathscr{O}_{F}&=O_{F}[t_{1}^{-1}].
\end{align*}
These are the valuation ring of rank $2$ and the discrete valuation ring of $F$ respectively. For example, when $F=\bb{F}_{q}((t_{1}))((t_{2}))$, we have
$$
O_{F}=\bb{F}_{q}[[t_{1}]]+t_{2}\bb{F}_{q}((t_{1}))[[t_{2}]], \mathscr{O}_{F}=\bb{F}_{q}((t_{1}))[[t_{2}]].
$$
Then let us write $G=SL_{2}(F)$ and consider the subgroups
$$
B=\begin{pmatrix}
O_{F}^{\times} & O_{F} \\
t_{1}O_{F} & O_{F}^{\times}
\end{pmatrix}\cap G
$$
and the monomial subgroups $N$. Then $W_{2}(A_{1})\simeq N/B\cap N$ is generated by three elements
$$
s=\begin{pmatrix}
0 & 1 \\
-1 & 0
\end{pmatrix} ~{\rm mod}~ B\cap N,
w_{1}=\begin{pmatrix}
0 & -t_{1}^{-1} \\
t_{1} & 0
\end{pmatrix} ~{\rm mod}~ B\cap N,
w_{2} =\begin{pmatrix}
0 & -t_{2}^{-1} \\
t_{2} & 0
\end{pmatrix} ~{\rm mod}~ B\cap N.
$$
The presentation of $W_{2}(A_{1})$ is given by
$$
W_{2}(A_{1})=\langle s,w_{1},w_{2}\mathrel{\vert} s^{2}=w_{1}^{2}=w_{2}^{2}=(sw_{1}w_{2})^{2}=1 \rangle.
$$
Let $K=BW(A_{1})B$. Then it's obvious that $K=SL_{2}(O_{F})$. We have the Bruhat decomposition and the Cartan decomposition:
$$
G=\coprod_{w\in W_{2}(A_{1})}BwB=\coprod_{v\in \bb{Z}^{2}_{\geq 0}(A_{1}^{+})}KvK.
$$
As a corollary of Proposition \ref{4.3.8}, we obtain the equality
$$
K(w_{2}Kw_{2})\cap (w_{1}Kw_{1})(w_{2}w_{1}Kw_{1}w_{2})=(K\cap w_{1}Kw_{1})(w_{2}w_{1}Kw_{1}w_{2}\cap w_{2}Kw_{2}).
$$
On the other hand, the triple $(G,B,N)$ is not a Tits system and $W_{2}(A_{1})$ is not a Coxeter group. Therefore a product of two Bruhat-cells can be an infinite number of Bruhat-cells. For example, we have
$$
C(w_{2})C(w_{2})=B\sqcup \coprod_{a\geq 0}C(w_{2}sw_{2}(w_{1}s)^{a})\sqcup\coprod_{b\in \bb{Z}}C(w_{2}(w_{1}s)^{b})\sqcup\coprod_{c\leq -1}C(s(w_{1}s)^{c})
$$ 
by simple calculations.
\end{exa}

\begin{prop}
For any $i,j=0,\ldots,n$, one has the decomposition
$$
\widehat{G}=\coprod_{w\in \widehat{W}}\widehat{\mf{B}}^{0}_{\mathbf{D}_{i}}w\widehat{\mf{B}}^{0}_{\mathbf{D}_{j}}.
$$
\end{prop}
\begin{proof}
Let $g\in \widehat{G}$. Let $\mf{C}$ and $\mf{D}$ be Babel sectors in $A$ of levels $i$ and $j$ with directions $\mathbf{D}_{i}$ and $\mathbf{D}_{j}$, respectively. 
By the definition of Babel buildings, there exists a Babel apartment $A'\in \mca{A}$ containing $g(\mf{C}')$ and $\mf{D}'$, where $\mf{C}'\subset \mf{C}$ and $\mf{D}'\subset \mf{D}$ are sub-sectors of the same levels. 
Since the group $\widehat{P}_{\mf{D}'}$ acts transitively on the set of apartments containing $\mf{D}'$, there exists an element $p\in \widehat{\mf{B}}^{0}_{\mathbf{D}_{j}}$ (specifically in $\widehat{P}_{\mf{D}'}$) such that $pA'=A$. 
Consequently, we have
$$
pg(\mf{C}'), \mf{D}' \subset A.
$$
Since the subgroup $\widehat{N}$ acts transitively on the set of $i$-level Babel sectors of $A$, we have $pg(\mf{C}')=n\mf{C}'$ for some $n\in \widehat{N}$. 
This implies that $n^{-1}pg$ fixes the sector $\mf{C}'$. 
Thus, we have
$$
g \in p^{-1}n \widehat{P}_{\mf{C}'} \subset \widehat{\mf{B}}^{0}_{\mathbf{D}_{j}} \widehat{N} \widehat{\mf{B}}^{0}_{\mathbf{D}_{i}}.
$$
(Note: Depending on the order of multiplication and definition of $p$, the inverse might be on the other side, but the set product remains the same).
Recall that any Babel sector contains an affine basis of any Babel apartment including it. Therefore, by applying Proposition \ref{4.2.1} (with $\Omega=\mf{C}'$ and $\Omega'=\mf{D}'$), we obtain the disjoint decomposition
$$
\widehat{G}=\coprod_{w\in \widehat{W}}\widehat{\mf{B}}^{0}_{\mathbf{D}_{i}}w\widehat{\mf{B}}^{0}_{\mathbf{D}_{j}}.
$$
\end{proof}

We call the above decomposition the \textbf{$(i,j)$-Kapranov decomposition}. This terminology is adopted from \cite[Proposition 1.2.3]{Kap}. The $(0,0)$-Kapranov decomposition corresponds to the Bruhat decomposition, while the $(0,n)$-Kapranov decomposition corresponds to the Iwasawa decomposition.

\begin{rem}
Reductive groups over higher-dimensional local fields are expected to admit Kapranov decompositions. If this holds, the buildings associated with these groups would constitute Babel buildings and possess the canonical metric structure described herein.
\end{rem}

\begin{exa}
Again we consider the example of $G=SL_{2}(F)$, where $F_{2}$ is a $2$-dimensional local field. Let us put
$$
S_{1}=\begin{pmatrix}
O_{F}^{\times} & \mathscr{O}_{F} \\
t_{2}\mathscr{O}_{F} & O_{F}^{\times} 
\end{pmatrix}\cap G, 
S_{2}=\begin{pmatrix}
1 & F \\
0 & 1
\end{pmatrix}.
$$
Then one has the $(0,1)$ and $(1,2)$-Kapranov decompositions
$$
G=BNS_{1}=S_{1}NS_{2}.
$$
\end{exa}

\subsection{Residues at vertices}

Let $p$ be a vertex of $X$ such that $r_{p}C=C$. We write
\begin{align*}
r_{p}\widehat{G}&=\widehat{P}_{r_{p}X}^{\dag}/\widehat{P}_{r_{p}X}, \\
r_{p}\widehat{B}&=\widehat{B}/\widehat{P}_{r_{p}X}, \\
r_{p}\widehat{N}&=\{g\in r_{p}\widehat{G}\mathrel{\vert} g(r_{p}A)=r_{p}A\}.
\end{align*}

Note that $\widehat{P}_{r_{p}X}$ is normal in $\widehat{P}_{\Omega}$ for any non-empty subset $\Omega\subset r_{p}X$. Indeed, take elements $g\in \widehat{P}_{\Omega},h\in \widehat{P}_{r_{p}X}$ and points $x\in r_{p}X, y\in \Omega$, we have
\begin{align*}
d_{X}(p,g(x))&\leq d_{X}(p,g(p))+d_{X}(p,x) \\
&\leq d_{X}(p,g(y))+d_{X}(g(p),g(y))+d_{X}(p,x) \\
&=2d_{X}(p,y)+d_{X}(p,x)\in {}^{(n-1)*}\bb{R}.
\end{align*}
Hence $g(r_{p}X)=r_{p}X$. In particular $h$ fixes $g(x)$. So we have $g^{-1}hg(x)=x$ namely $g^{-1}hg\in \widehat{P}_{r_{p}X}$.

\begin{prop}
The group $r_{p}\widehat{G}$ acts on the residue $r_{p}X$ strong transitively.
\end{prop}
\begin{proof}
Let $r_{p}A_{1},r_{p}A_{2}\in r_{p}\mca{A}$ be two Babel apartments of level $n-1$ and $C_{1},C_{2}$ chambers of them respectively. We should an element $g\in r_{p}\widehat{G}$ such that $g(r_{p}A_{1},C_{1})=(r_{p}A_{2},C_{2})$. By definition, there exist $A_{1},A_{2}\in \mca{A}$ such that their residues at $p$ are $r_{p}A_{1},r_{p}A_{2}$ respectively. By assumption, we can take an element $g\in \widehat{G}$ such that $g(A_{1},C_{1})=(A_{2},C_{2})$. For any $c\in C_{1}$ and $x\in r_{p}X$, since $g(c)\in C_{2}\subset r_{p}X$ and
\begin{align*}
d_{X}(p,g(x))&\leq d_{X}(p,g(p))+d_{X}(p,x) \\
&\leq d_{X}(p,g(c))+d_{X}(g(p),g(c))+d_{X}(p,x)\in {}^{(n-1)*}\bb{R},
\end{align*}
one has $g(x)\in r_{p}X$. Hence $g$ belongs to $\widehat{P}^{\dag}_{r_{p}X}$ and we have $g(r_{p}A_{1},C_{1})=(r_{p}A_{2},C_{2})$.

\end{proof}

\begin{exa}
Let $F=F_{2}$ be a $2$-dimensional local field and $F_{1}$ its residue field. Let $X$ be the Babel tree of $SL_{2}(F)$ and $p$ the vertex fixed by $K=SL_{2}(O_{F}).$
Note that 
$$
1+t_{2}M_{2}(\mathscr{O}_{F})=\bigcap_{m\geq 1} 1+ t_{1}^{m}M_{2}(O_{F}).
$$
and the subgroup $1+t_{1}^{m} M_{2}(O_{F})$ is the fixer the set $\{x\in r_{p}X\mathrel{\vert} d_{\rm gal}(p,x)\leq m\}$, where $d_{\rm gal}(p,x)$ is the length of a minimal gallery from $p$ to $x$. Since the affine building $r_{p}X$ is gallery-connected, the fixer of $r_{p}X$ equals to $1+t_{2}M_{2}(\mathscr{O}_{F})$. Recall the action of $SL_{2}(\mathscr{O}_{F})\subset G$ on $X$. For any $x\in r_{p}X$ and $k\in SL_{2}(\mathscr{O}_{F})$, we should show that
$$
d_{X}(p,k(x))\in \bb{R}.
$$
Indeed we have
$$
d_{X}(p,k(x))\leq d_{X}(p.k(p))+d_{X}(p,x),
$$
we assume that $p=x$. By the Cartan decomposition, we may assume that $k={\rm diag}(t_{1}^{m},t_{1}^{-m})$ for some $m\in \bb{Z}$. Then $d_{X}(p,k(p))=|m|.$ Thus we obtain $r_{p}SL_{2}(F_{2})=SL_{2}(F_{1})$.
\end{exa}

\begin{rem}
As the above example, The word ''residue'' comes from the residue fields of higher dimensional local fields. Geometrically speaking, this is an operation that narrows the field of vision and focuses on a specific area around the vertex.
\end{rem}

The following several claims are residue versions of propositions already proved. So we omit proofs.

\begin{prop}
We have the Bruhat decomposition
$$
r_{p}\widehat{G}=\coprod_{w\in r_{p}\widehat{W}}r_{p}\widehat{B}w r_{p}\widehat{B},
$$
where $r_{p}\widehat{W}=r_{p}\widehat{N}/r_{p}\widehat{B}\cap r_{p}\widehat{N}$.
\end{prop}
\begin{proof}
Omit.
\end{proof}
\begin{prop}
We have the Cartan decomposition
$$
r_{p}\widehat{G}=\coprod_{v\in \widehat{V}_{r_{p}\mb{D}}} (r_{p}K)v(r_{p}K),
$$
where $r_{p}\mb{D}$ is a fundamental domain of $W(\Phi)$-action on $r_{p}A$ and $r_{p}K=r_{p}\widehat{B}W(\Phi)r_{p}\widehat{B}$.
\end{prop}
\begin{proof}
Omit.
\end{proof}

\begin{prop}
For any $i,j=0,\ldots,n-1$, we have the $(i,j)$-Kapranov decomposition
$$
r_{p}\widehat{G}=\coprod_{w\in r_{p}\widehat{W}} \widehat{\mf{B}}^{0}_{r_{p}\mb{D}_{i}}w\widehat{\mf{B}}^{0}_{r_{p}\mb{D}_{j}}.
$$
\end{prop}
\begin{proof}
Omit.
\end{proof}

\begin{prop}
Let $x,y,z,u$ be four points in $r_{p}A$ such that $y,z\in [x,u]$ with $d_{r_{p}A}(x,y)\leq d_{r_{p}A}(x,z)$. Then one has
$$
r_{p}\widehat{P}_{x}r_{p}\widehat{P}_{u}\cap r_{p}\widehat{P}_{y}r_{p}\widehat{P}_{z}=(r_{p}\widehat{P}_{x}\cap r_{p}\widehat{P}_{y})(r_{p}\widehat{P}_{z}\cap r_{p}\widehat{P}_{u}),
$$ 
where we put $r_{p}\widehat{P}_{x}=\widehat{P}_{x}/\widehat{P}_{r_{p}X}$.
\end{prop}
\begin{proof}
Omit.
\end{proof}

\end{document}